\documentclass{amsproc}%
\usepackage{amsfonts}
\usepackage{amsmath}
\usepackage{amsthm}
\usepackage[v2]{xy}
\usepackage{amssymb}
\usepackage{graphicx}
\usepackage{enumerate}
\usepackage{color}
\usepackage{setspace}

\newtheorem{theorem}{Theorem}[section]
\newtheorem{corollary}[theorem]{Corollary}
\newtheorem{lemma}[theorem]{Lemma}
\newtheorem{proposition}[theorem]{Proposition} 
\theoremstyle{definition}
\newtheorem{definition}[theorem]{Definition}
\theoremstyle{example}

\numberwithin{equation}{section}
\theoremstyle{remark}

\numberwithin{equation}{section}

\newcommand{\dd}{\,\textrm{d}}

\begin{document}
\allowdisplaybreaks[4]

\renewcommand{\baselinestretch}{1.2}  \large \normalsize

\title{Additive spectra of the $\frac14$ Cantor measure}

\author[P.E.T. Jorgensen]{Palle E. T. Jorgensen}
\address[Palle E.T. Jorgensen]{Department of Mathematics, The University of Iowa, Iowa
City, IA 52242-1419, U.S.A.}
\email{palle-jorgensen@uiowa.edu}
\urladdr{http://www.math.uiowa.edu/\symbol{126}jorgen/}
\author[K.A. Kornelson]{Keri A. Kornelson}
\address[Keri Kornelson]{Department of Mathematics, The University of Oklahoma, Norman, OK, 73019, U.S.A.}
\email{kkornelson@ou.edu}
\urladdr{http://www.math.ou.edu/\symbol{126}kkornelson/}
\author[K.L. Shuman]{Karen L. Shuman}
\address[Karen Shuman]{Department of Mathematics and Statistics,
Grinnell College, Grinnell, IA 50112-1690, U.S.A.}
\email{shumank@math.grinnell.edu}
\urladdr{http://www.math.grinnell.edu/\symbol{126}shumank/}

\thanks{The second author was supported in part by grant \#244718 from The Simons Foundation.}

\begin{abstract}
In this paper, we add to the characterization of the Fourier spectra for Bernoulli convolution measures.  These measures are supported on Cantor subsets of the line.  We prove that performing an odd additive translation to half the canonical spectrum for the $\frac14$ Cantor measure always yields an alternate spectrum.   We call this set an additive spectrum.  The proof works by connecting the additive set to a spectrum formed by odd multiplicative scaling. 
\end{abstract}
\date{\today}
\maketitle


\thispagestyle{empty}

 \section{Introduction}
 
  Traditional Cantor sets are generated by iterations of an operation of down-scaling by fractions which are powers of a fixed positive integer. For each iteration in the process, we leave gaps.  For example, the best-known ternary Cantor set is formed by scaling down by $\frac13$ and leaving a single gap in each step.  An associated Cantor measure $\mu$ is then obtained by the same sort of iteration of scales, and, at each step, a renormalization. In accordance with classical harmonic analysis, these measures may be seen to be infinite Bernoulli convolutions.

   Our present analysis is motivated by earlier work, beginning with \cite{JoPe98}.   We consider recursive down-scaling by  $\frac{1}{2n}$ for $n \in \mathbb{N}$ and leave a single gap at each iteration-step. It was shown in \cite{JoPe98}  that the associated Cantor measures $\mu_{\frac{1}{2n}}$  have the property that $L^2(\mu_{\frac{1}{2n}})$ possesses orthogonal Fourier bases of complex exponentials (i.e. Fourier ONBs).  More recently, it was shown in \cite{Dai12} that the scales $\frac{1}{2n}$ are the \textit{only} values that generate measures with Fourier bases.  
   
     Given a fixed Cantor measure $\mu$,  a corresponding set of frequencies  $\Gamma$ of exponents in an ONB is said to be a \textit{spectrum} for $\mu$.  For example, in the case of recursive scaling by powers of $\frac14$ , i.e. $n= 2$, a possible spectrum $\Gamma$ for $L^2(\mu)$  has the form  $\Gamma$ as shown below in Equation \eqref{eqn:Gamma}.   A spectrum for a Cantor measure turns out to be a \textit{lacunary} (in the sense of Szolem Mandelbrojt) set of integers or half integers. We direct the interested reader to \cite{Kah85} regarding lacunary series and their Riesz products.

   When $n$ and $\mu$  are fixed, we now become concerned with the possible variety of spectra.  Given $\Gamma$ some canonical choice of spectrum for $\mu$, then one possible way to construct a new Fourier spectrum for $L^2(\mu)$ is to scale by an odd positive integer $p$ to form a set $p\Gamma$.  While for some values of $p$ this scaling produces a spectrum, it is known that other values of $p$ do not yield spectra.  
     
 This particular question is intrinsically multiplicative: Since $\mu$ is an infinite Bernoulli convolution,  the ONB questions  involve consideration of infinite products of the Riesz type.  Despite this intuition, we show here (Theorem \ref{thm:main}) that there is a connection between this multiplicative construction and a construction of new ONBs with an additive operation.  We are then able to produce even more examples of these additive-construction spectra. 
 
\section{Background}
Throughout this paper, we consider the Hilbert space $L^2(\mu_{\frac14})$ where $\mu_{\frac14}$ is the $\frac14$-Bernoulli convolution measure.  This measure has a rich history, dating back to work of Wintner and Erd\H{o}s \cite{W35, E39, E40}.  More recently, Hutchinson \cite{Hut81} developed a construction of Bernoulli measures via iterated function systems (IFSs).  The measure $\mu_{\frac14}$ is supported on a Cantor subset $X_{\frac14}$ of $\mathbb{R}$ which entails scaling by $\frac14$.  In 1998, Jorgensen and Pedersen \cite{JoPe98}  discovered that the Hilbert space $L^2(\mu_{\frac14})$ contains a Fourier basis --- an orthonormal basis of exponential functions --- and hence allows for a Fourier analysis. 

For ease of notation, throughout this paper we will write $e_t$ for the function $e^{2\pi i t \cdot}$ and for a discrete set $\Gamma$ we will write $E(\Gamma)$ for the collection of exponentials $\{e_{\gamma}\,:\, \gamma \in \Gamma\}$. 

 There is a self-similarity inherent in the $\frac14$-Bernoulli convolution \begin{equation}\label{eqn:ss-measure} 
\int_{X_{\frac14}} f \dd \mu_{\frac14} = \frac12 \int_{X_{\frac14}} f\Bigr(\frac14(x+1)\Bigr) \dd \mu_{\frac14}(x) + \frac12 \int_{X_{\frac14}} f\Bigr(\frac14(x-1)\Bigr) \dd \mu_{\frac14}(x)
\end{equation}
which yields an infinite product formulation for $\widehat{\mu}_{\frac14}$:
\begin{equation}\label{eqn:infproduct}
\widehat{\mu}_{\frac14}(t) = \int_{X_{\frac14}}e^{2\pi i t x} \dd \mu_{\frac14}(x) = \prod_{k=1}^{\infty} \cos\Bigr(\frac{2\pi t}{4^k}\Bigr).
\end{equation}

Exponential functions $e_{\gamma}$ and $e_{\xi}$ are orthogonal when \[ \langle e_{\gamma}, e_{\xi} \rangle = \widehat{\mu}_{\frac14}(\gamma - \xi) = 0.\]  A collection of exponential functions $E(\Gamma)$ indexed by the discrete set $\Gamma$ is an orthonormal basis for $L^2(\mu_{\frac14})$ exactly when the function \begin{equation}\label{eqn:spectral} c_{\Gamma}(t) := \sum_{\gamma \in \Gamma} |\langle e_t, e_{\gamma} \rangle |^2 = \sum_{\gamma \in \Gamma} \prod_{k=1}^{\infty} \cos^2\Bigr(\frac{2 \pi (t-\gamma)}{4^k} \Bigr) \end{equation} is the constant function $1$.  We call the function $c_{\Gamma}$ the \textit{spectral function} for the set $\Gamma$.

The Fourier basis for $\mu_{\frac14}$ constructed in \cite{JoPe98} is the set $\{e^{2\pi i \gamma \cdot}\,:\, \gamma \in \Gamma\}$, where 
\begin{equation}\label{eqn:Gamma} \Gamma = \left\{ \sum_{i=0}^m a_i4^i \,:\, m \textrm{ finite}, a_i \in \{0,1\} \right\} = \{0, 1, 4, 5, 16, 17, 20 \ldots\}. 
\end{equation}

If $E(\Gamma)$ is an orthonormal basis (ONB) for $L^2(\mu_{\frac14})$, we say that $\Gamma$ is a \textit{spectrum} for $\mu_{\frac14}$.  It is straightforward to show that if $\Gamma$ is a spectrum for $\mu_{\frac14}$ and $p$ is an odd integer, then $E(p\Gamma)$ is an orthogonal collection of exponential functions.  In many cases, we find that $E(p\Gamma)$ is actually another ONB \cite{LaWa02,DuJo09,JKS11}.  This is rather surprising, or at least very different behavior from the usual Fourier analysis on an interval with respect to Lebesgue measure.   

We often refer to the spectrum in Equation \eqref{eqn:Gamma} as the \textit{canonical spectrum} for $L^2(\mu_{\frac14})$, while other spectra for the same measure space can be called \textit{alternate spectra}. 
\section{Isometries} 

In this section, we describe two naturally occurring isometries on $L^2(\mu_{\frac14})$ which are defined via their action on the canonical Fourier basis $E(\Gamma)$.  Observe from Equation \eqref{eqn:Gamma} that $\Gamma$ satisfies the invariance equation \[ \Gamma = 4\Gamma \sqcup (4\Gamma +1),\] where $\sqcup$ denotes the disjoint union.  We then define 

\begin{eqnarray}\label{eqn:s0s1} S_0 &:& e_{\gamma} \mapsto e_{4\gamma} \\ \nonumber S_1&:& e_{\gamma} \mapsto e_{4\gamma +1} \qquad \textrm{for all } \gamma \in \Gamma. \end{eqnarray} 
Since $S_0$ and $S_1$ map the ONB elements into a proper subset of the ONB, they are proper isometries.  Therefore, for $i=0,1$ we have $S_i^*S_i = I$ and $S_iS_i^*$ is a projection onto the range of the respective operator. The adjoints of $S_0, S_1$ are readily computed (see \cite{JKS12a} for details):   \begin{equation}\label{Eqn:S_0^*}
S_0^*e_{\gamma} = \left\{ \begin{matrix} e_{\frac{\gamma}{4}} & \textrm{when } \gamma \in 4\Gamma \\
0 & \textrm{otherwise} \end{matrix} \right.
\end{equation}
and
\begin{equation}\label{Eqn:S_1^*} 
S_1^*e_{\gamma} =  \left\{ \begin{matrix} e_{\frac{\gamma- 1}{4}} & \textrm{when } \gamma \in 1 + 4\Gamma \\
0 & \textrm{otherwise.} \end{matrix} \right.
\end{equation}

It is shown in \cite[Section 2]{JKS12c} that the definitions of $S_0$ and $S_1$ extend to all $e_n$ for $n \in \mathbb{Z}$, i.e. 
\begin{equation}\label{eqn:s0s1Z} S_0: e_n \mapsto e_{4n} \textrm{\; and \;} S_1: e_n \mapsto e_{4n+1} \quad \forall n \in \mathbb{Z}.\end{equation}

For every integer $N > 1$, there is a $C^*$-algebra with $N$ generators called the Cuntz algebra, which we denote by $\mathcal{O}_N$ \cite{Cun77}.   We will describe representations of $\mathcal{O}_2$ which are generated by two isometries on $L^2(\mu_{\frac14})$ satisfying the conditions below.
 
 \begin{definition}\label{Defn:CuntzRel}  We say that isometry operators $T_0, T_1$ on $L^2(\mu_{\frac14})$ satisfy \textit{Cuntz relations}
if 
\begin{enumerate}  
\item $T_0T_0^* + T_1T_1^* = I$,  
\item $T_i^*T_j = \delta_{i,j}I$ \;\; for $i,j = 0, 1$.
\end{enumerate}
When these relations hold, $\{T_0,T_1\}$ generate a representation of the Cuntz algebra $O_2$.
\end{definition}

From \cite{BrJo99,JKS12a}, we know that $S_0$ and $S_1$ defined in Equation \eqref{eqn:s0s1} satisfy the Cuntz relations for $k=2$, hence yield a representation of the Cuntz algebra $\mathcal{O}_2$ (in fact, an irreducible representation) within the algebra of bounded operators $\mathcal{B}(L^2(\mu_{\frac14}))$. 


 \section{Spectral function decompositions}
As we mentioned above, given a spectrum $\Gamma$, the frequencies $p\Gamma$, for $p$ an odd integer, generate an orthonormal collection of exponential functions in $L^2(\mu)$.  Given $\Gamma$ from Equation \eqref{eqn:Gamma}, one question of interest is the characterization of the odd integers $p$ for which the scaled spectrum $p\Gamma$ generates an ONB.  As a means of exploring this question, we let $U_p$ be the operator \begin{equation}\label{eqn:Up} U_p: e_{\gamma} \mapsto e_{p\gamma} \qquad \forall \gamma \in \Gamma.\end{equation}  Since $U_p$ maps an ONB to an orthonormal collection, $U_p$ is an isometry and is unitary if and only if $E(p\Gamma)$ is an ONB.

The following lemmas provide useful relationships between the isometries $S_0$, $S_1$, and $U_p$. 
\begin{lemma}\label{lem:vp} Let $S_0$ and $S_1$ be the isometry operators from Equation  \eqref{eqn:s0s1}.  If $\rho$ is a $*$-automorphism on $\mathcal{B}(L^2(\mu_{\frac14}))$, then the operator \[ W = \rho(S_0)S_0^* + \rho(S_1)S_1^*\] is unitary.

\end{lemma}
\begin{proof} 
Assume $\rho$ is a $*$-automorphism.  The Cuntz relations on $S_0$ and $S_1$ give 
\begin{eqnarray*} WW^* &=& \Bigr(\rho(S_0)S_0^* + \rho(S_1)S_1^*\Bigr) \Bigr(\rho(S_0)S_0^* + \rho(S_1)S_1^*\Bigr)^*\\ &=& 
 \Bigr(\rho(S_0)S_0^* + \rho(S_1)S_1^*\Bigr)\Bigr(S_0\rho(S_0^*) + S_1\rho(S_1^*) \Bigr)\\ &=& \rho(S_0)S_0^*S_0 \rho(S_0^*) + \rho(S_0)S_0^*S_1\rho(S_1^*) \\& & \quad + \rho(S_1)S_1^*S_0\rho(S_0^*) + \rho(S_1)S_1^*S_1\rho(S_1^*) \\ &=& 
 \rho(S_0)\rho(S_0^*) + \rho(S_1)\rho(S_1^*)\\ &=& \rho(S_0S_0^* + S_1S_1^*)\\ &=& \rho(I) = I
 \end{eqnarray*}

A similar computation proves that $W^*W = I$, hence $W$ is unitary.
\end{proof}

\begin{lemma}\label{lem:US1} Let $M_k$ be the multiplication operator $M_kf = e_kf$.  Given $p \in \mathbb{N}$ such that $U_p$ is unitary, we define the map $\widetilde{\alpha}(X) = U_pXU_p^*$ on $\mathcal{B}(L^2(\mu_{\frac14}))$.  Then $\widetilde{\alpha}(S_0) = S_0$ and  $\widetilde{\alpha}(S_1) = M_{p-1}S_1$.  
\end{lemma}
\begin{proof}  It was proved in \cite{JKS12a} that  $U_p$ commutes with $S_0$ for all odd $p$, so $\widetilde{\alpha}(S_0) = S_0$.  Since $U_p$ is unitary, we have $U_p^*U_p =U_pU_p^*= I$.  We prove that $M_{p-1}S_1U_p =U_pS_1$, which is thus equivalent to the statement of the lemma. 
\begin{eqnarray*}  U_pS_1e_{\gamma} &=& U_pe_{4\gamma +1}\\ &=& e_{4p\gamma + p} , \quad \textrm{and}\\ M_{p-1}S_1U_pe_{\gamma} &=& M_{p-1}S_1e_{p\gamma}\\ &=& M_{p-1}e_{4p\gamma +1} \quad \textrm{by extension of } S_1 \textrm{ to } \mathbb{N} \\ &=& e_{4p\gamma + p} \end{eqnarray*}
Therefore, \[\widetilde{\alpha}(S_1) = U_pS_1U_p^* = M_{p-1}S_1U_pU_p^* = M_{p-1}S_1.\]

\end{proof}

 We now discover a connection between the scaled spectrum $p\Gamma$ and what we call an \textit{additive spectrum} $E(4\Gamma) \cup E(4\Gamma + p)$.  It will turn out that this connection tells us more about the additive spectra than the scaled spectra.

\begin{theorem}\label{thm:main}
Given any odd natural number $p$, if $E(p\Gamma)$ is an ONB then $E(4\Gamma) \cup E(4\Gamma + p)$ is also an ONB.
\end{theorem} 

\begin{proof}
Since $E(p\Gamma)$ is an ONB, we have that the operator $U_p$ from Equation \eqref{eqn:Up}  is a unitary operator.  We define the map on $\mathcal{B}(L^2(\mu_{\frac14})))$ \begin{equation}\label{eqn:alphap} \widetilde{\alpha}(X) = U_pXU_p^*.\end{equation}
Since $U_p$ is unitary, it is straightforward to verify that $\widetilde{\alpha}$ is a $*$-automorphism on $\mathcal{B}(L^2(\mu_{\frac14}))$.

If we apply $\widetilde{\alpha}$ to our operators $S_0$ and $S_1$, we have by Lemma \ref{lem:US1},
\[ \widetilde{\alpha}(S_0) = S_0 \quad \textrm{and} \quad \widetilde{\alpha}(S_1) = U_pS_1U_p^* = M_{p-1}S_1.\]  Define the operator \[\widetilde{W} := \widetilde{\alpha}(S_0)S_0^* + \widetilde{\alpha}(S_1)S_1^* = S_0S_0^* +  M_{p-1}S_1S_1^*.\]  Then $\widetilde{W}$ is unitary by Lemma \ref{lem:vp}.

We see that if $\gamma \in 4\Gamma$, i.e.  $\gamma = 4\gamma'$ for some $\gamma' \in \Gamma$, that $\widetilde{W}e_{\gamma} = \widetilde{W}S_0e_{\gamma'} = e_{\gamma}$ since $S_0^*S_0 = I$ and $S_1^*S_0 = 0$ by the Cuntz relations.   Similarly, if $\gamma \in 4\Gamma+1$, hence  $\gamma = 4\gamma'+1$ for some $\gamma' \in \Gamma$, then $\widetilde{W}e_{\gamma} = \widetilde{W}S_1e_{\gamma'} = M_{p-1}S_1e_{\gamma'} = e_{4\gamma'+p}$.  In fact, $\widetilde{W}$ maps $E(4\Gamma + 1)$ bijectively onto $E(4\Gamma + p)$.  Therefore, since $\widetilde{W}$ is unitary, we can conclude that $E(4\Gamma) \cup E(4\Gamma + p)$ is an ONB for $L^2(\mu_{\frac14})$.
\end{proof}
 
  We now address the spectral functions---recall Equation \eqref{eqn:spectral}---for our additive sets.   We can use the splitting $\Gamma = (4\Gamma) \cup (4\Gamma +1)$ to divide the spectral function for $\Gamma$ into the corresponding terms  \[ c_{\Gamma}(t) = \sum_{\gamma \in \Gamma}\prod_{k=1}^{\infty} |\widehat{\mu}(t-4\gamma)|^2 + \sum_{\gamma \in \Gamma}\prod_{k=1}^{\infty} |\widehat{\mu}(t-4\gamma-1)|^2.\]
Denote the sums on the right-hand side of the equation above by $c_0(t)$ and $c_1(t)$ respectively.  More generally, denote \begin{equation} c_m(t) = \sum_{\gamma \in \Gamma}\prod_{k=1}^{\infty} |\widehat{\mu}(t-4\gamma-m)|^2.\end{equation}

 \begin{proposition}\label{prop:per} The function $c_1$ is $2$-periodic.
\end{proposition}
\begin{proof}

By Theorem \ref{thm:main} the sets $(4\Gamma) \cup (4\Gamma + 5)$ and $(4\Gamma) \cup (4\Gamma +7)$ are both spectra for $\mu$ --- this follows because it is known (see, for example, \cite{LaWa02}) that the scaled sets $5\Gamma$ and $7\Gamma$ are spectra.
We therefore have
\[ 1 = c_0(t) + \sum_{\gamma \in \Gamma}\prod_{k=1}^{\infty} |\widehat{\mu}(t-4\gamma-5)|^2 = c_0(t) + \sum_{\gamma \in \Gamma}\prod_{k=1}^{\infty} |\widehat{\mu}(t-4\gamma-7)|^2.\]  Using the fact that the set $\Gamma$ itself is also a spectrum, we have \[ 1 = c_0(t) + c_1(t) = c_0(t) + c_5(t) = c_0(t) + c_7(t)\] for all $t \in \mathbb{R}$.   Hence \[ c_1(t) = c_5(t) = c_7(t) \quad \forall t \in \mathbb{R}.\]   But we also observe that $c_5(t) = c_1(t-4)$ and $c_7(t) = c_1(t-6)$, so the function $c_1$ is both $4$-periodic and $6$-periodic, hence is $2$-periodic.
\end{proof}

\begin{figure}[h]
\begin{center}
\setlength\fboxsep{0pt}
  \setlength\fboxrule{0.5pt}
    \fbox{
\includegraphics{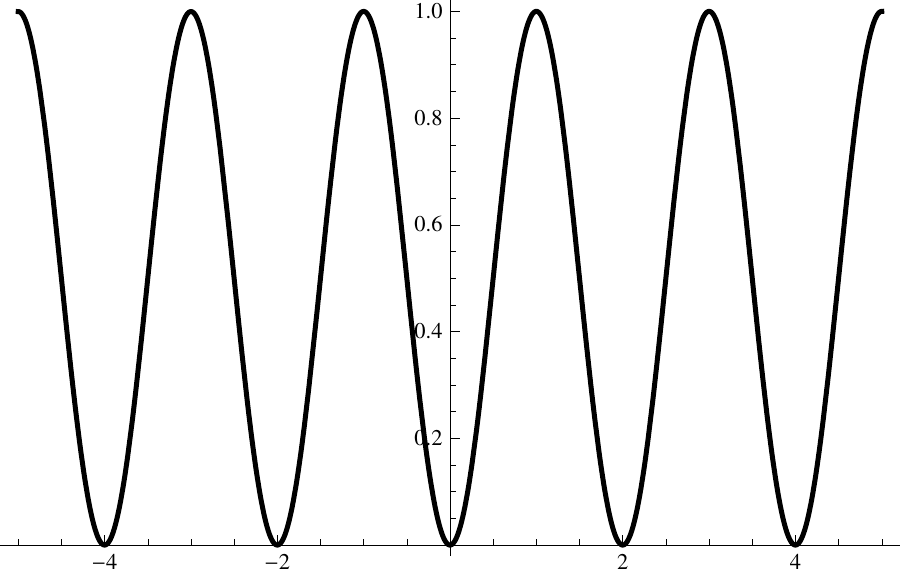}
}
\caption{Numerical estimate of $c_1$, with $16$ factors in the product and $128$ terms in the sum.}

\end{center}\end{figure}

\begin{corollary}
The function $c_0$ is $2$-periodic.
\end{corollary}

We next observe that Theorem \ref{thm:main} is a stepping stone to the following result.  

\begin{theorem}\label{thm:dhs}  Given any odd integer $p$, the set $E[(4\Gamma) \cup (4\Gamma+p)]$ is an ONB for $L^2(\mu)$.
\end{theorem}
 
 \begin{proof}  This is a direct result of Proposition \ref{prop:per}.  The spectral function for$E[(4\Gamma) \cup (4\Gamma+p)]$  can be written in the two parts \[ \sum_{\gamma \in \Gamma} \prod_{k=1}^{\infty} |\widehat{\mu}(t-4\gamma)|^2 + \sum_{\gamma \in \Gamma} \prod_{k=1}^{\infty} |\widehat{\mu}(t-(4\gamma+p))|^2  .\]  When $p=1$, we have the canonical ONB in the $\frac14$ case.  Otherwise, using the 2-periodicity of $c_0$, we have 
 
 \begin{eqnarray*}  
 c_0(t) + c_p(t) &=& \sum_{\gamma \in \Gamma} \prod_{k=1}^{\infty} |\widehat{\mu}(t-4\gamma)|^2 + \sum_{\gamma \in \Gamma} \prod_{k=1}^{\infty} |\widehat{\mu}(t-(4\gamma+p))|^2\\ &=&  \sum_{\gamma \in \Gamma} \prod_{k=1}^{\infty} |\widehat{\mu}(t-(p-1)-4\gamma)|^2 + \sum_{\gamma \in \Gamma} \prod_{k=1}^{\infty} |\widehat{\mu}(t-(4\gamma+1)-(p-1))|^2  \\&=& c_0(t-p+1) + c_1(t-p+1) \equiv 1.
 \end{eqnarray*}
 
 Since the spectral function is identically $1$, the set  $E[(4\Gamma) \cup (4\Gamma+p)]$ is an ONB for $L^2(\mu)$.
\end{proof} 
 
 \section*{Acknowledgements}
 The authors would like to thank Allan Donsig for helpful conversations while writing an earlier version of this work.
 
  We mention here that the existence of the spectra that we call the \textit{additive spectra} for $\mu_{\frac14}$ is not new.  They are among the examples described, from a different perspective, in Section 5 of  \cite{DHS09}.

\bibliographystyle{alpha}

\end{document}